\documentclass[11pt,a4paper,reqno]{amsart}

\usepackage[latin1]{inputenc}
\usepackage[english]{babel}
\usepackage{amsmath}
\usepackage[numbers]{natbib}
\usepackage{graphicx}
\usepackage{amsmath,bm}
\usepackage{booktabs}
\usepackage{graphicx}
\usepackage{multirow}
\usepackage{setspace}
\usepackage{braket}
\newtheorem{rmk}{Remark}
\newtheorem{thm}{Theorem}
\usepackage[hmargin=3cm,vmargin={3cm,4cm}]{geometry}

\usepackage{etoolbox}
\newcommand*{\affaddr}[1]{#1} 
\newcommand*{\affmark}[1][*]{\textsuperscript{#1}}

\begin{document}

\title{Analysis of a spherical free boundary problem modelling granular biofilms}

\author[F.~Russo et al.]{F.~Russo\protect\affmark[1] \and M.R.~Mattei\affmark[1] \and A.~Tenore\affmark[1] \and B.~D'Acunto\affmark[1] \and V.~Luongo\affmark[1] \and L.~Frunzo\affmark[1]}

 \maketitle

{\footnotesize
\begin{center}
\affaddr{\affmark[1]Department of Mathematics and Applications "Renato Caccioppoli", University of Naples Federico II, Via Cintia 1, Montesantangelo, 80126, Naples, Italy}\\
Corresponding author: M.R.~Mattei,
\texttt{mariarosaria.mattei@unina.it}
\end{center}}

\begin{abstract}
A free boundary value problem related to the genesis of multispecies granular biofilms is presented. The granular biofilm is modelled as a spherical free boundary domain with radial symmetry. The proposed model is conceived in the framework of continuum mechanics and consists of: nonlinear hyperbolic PDEs which model the advective transport and growth of attached species that constitute the granular biofilm matrix; semilinear elliptic PDEs which govern the diffusive transport and conversion of nutrients; and semilinear elliptic PDEs describing the invasion phenomena and conversion of planktonic cells suspended in the surrounding environment. The evolution of the free boundary is governed by an ODE accounting for microbial growth, attachment, and detachment. By using the method of characteristics, the system of equations constituting the granular biofilm model is converted into an equivalent integral system. Existence and uniqueness of solutions are discussed and proved for the attachment regime using the fixed point theorem.
\end{abstract}

\maketitle

\section{Introduction} \label{n1}
\
In most natural and human environments, microorganisms do not live as pure cultures of dispersed single cells, but are frequently embedded in a self-produced matrix of extracellular polymeric substances ($EPS$), forming complex, dense and compact biofilms \cite{ flemming2010biofilm, flemming2016biofilms}. Many species from several trophic groups may coexist in such microbial consortia and interact with each other through synergistic and antagonistic activities. Differently from free-swimming planktonic cells, bacteria living in a biofilm benefit from interspecies cooperation, showing higher resistance capacities to toxic substances and antibiotics \cite{mattei2018continuum}. In nature, biofilms predominantly growth attached to living or solid suitable surfaces \cite{palmer2007bacterial}. Nevertheless, in some engineered systems microorganisms aggregation can occur due to the self-immobilization of cells into approximately spherical-shaped biofilms \cite{trego2020growth}. Such aggregates share the typical properties of conventional planar biofilms in terms of density and composition, but they have a spherical structure and are able to freely move within the surrounding liquid medium. Nevertheless, the geometry and free movement of granules limits external boundary layer resistances, promoting mass transfer of nutrients towards microorganisms \cite{baeten2019modelling}.

Biofilm formation is a dynamic and complex process, which consists of several stages resulting from physical (substrates transport, invasion, attachment, detachment, etc.) and biochemical (microbial growth, substrates conversion, etc.) factors \cite{mattei2018continuum}, both in the case of planar and granular biofilms. The process leading to the formation of granular biofilms is known as granulation. Specifically, the term \textit{de novo} granulation denotes the process initiated by individual microbial cells, differently from the case in which granulation proceeds from already formed granules. Pioneer species in planktonic form aggregate thanks to their self-immobilization abilities and attachment properties. Once attached, bacteria proliferate and form the first sessile microbial colony, which expands over time as result of the microbial metabolic activities. In this first stage, the biofilm is characterized by a high level of diversity, as several extremely heterogeneous microenvironments may form inside it. This favours invasion phenomena of new microbial species, which benefit from the biofilm protective environment. Thus, the biofilm is colonized by motile planktonic cells which penetrate the biofilm matrix and proliferate as a new sessile biomass, where ideal conditions for their metabolic activities occur. During the maturation stage, external shear forces, nutrients depletion and biomass decay lead to the detachment of biofilm clusters from the biofilm to the surrounding medium. Conversely, detachment can be initiated internally, causing the dispersion of individual planktonic cells.

The relevance of this topic in engineering, biological and industrial fields has stimulated a growing interest in biofilms modelling. Indeed, numerous modelling works on multispecies biofilms growth and formation have been proposed for the planar \cite{eberl2001new, emerenini2015mathematical, alpkvista2007multidimensional,nat1,nat2} and granular \cite{feldman2017modelling, odriozola2016modeling, volcke2010effect} case. Nevertheless, several significant aspects of biofilms growth are not exhaustively considered by existing models. Indeed, most models completely neglect the attachment process by pioneering species in the initial phase of the biofilm formation. Only the $1D$ biofilm model and the granular biofilm model proposed by D'Acunto et al. \cite{d2019free,d2021free} and Tenore et al. \cite{tenore2021multiscale} focus on the \textit{de novo} formation of biofilms, by considering a vanishing initial domain.

In this work, the granular biofilm model describing the \textit{de novo} granulation process of a multispecies granular biofilm in the framework of continuous models is discussed \cite{tenore2021multiscale}. Modelling the \textit{de novo} granulation process does not require the definition of any initial conditions, as the initial size of the biofilm domain is assigned to zero and the microbial species composition is modelled according to the environmental conditions. The granular biofilm is modelled as a spherical free boundary domain under the assumption of radial symmetry. Its formation and expansion are governed by microbial growth, invasion phenomena by planktonic cells, attachment and detachment processes. As highlighted by recent works \cite{d2019free,d2021free}, the presence of free cells suspended in the surrounding liquid medium play a prominent role in the biofilm formation and growth. Indeed, attachment of pioneer planktonic cells initiates the granulation process, while the stage of biofilm development and maturation is strongly affected by colonization process mediated by microbial species not initially present within the biofilm. Therefore, in this work attachment flux is modelled as a continuous deterministic process, which linearly depends on the concentrations of the planktonic species in the surrounding medium each of which is characterized by a specific attachment velocity \cite{d2019free,d2021free}. Instead, invasion process is modelled by considering an additional reaction term, which depends on the concentration of planktonic species within the biofilm \cite{d2015modeling}. Specifically, the model considers two state variables representing the planktonic and sessile phenotypes and the conversion from the first to the second during the biofilm granulation process. Furthermore, a qualitative analysis of the spherical free boundary problem for the initial phase of multispecies granular biofilm growth is here presented for the first time following the approach proposed in D'Acunto et al. \cite{d2019free,d2021free}. 
 
The work is organized as follows. Section \ref{n2} discusses the spherical free boundary value problem by presenting all assumptions, variables, equations, and initial and boundary
conditions. The spherical free boundary evolution represented by the granular biofilm is governed by a first order differential equation. The growth of the attached species is governed by non linear hyperbolic partial differential equations. During the first instants of the biofilm formation, the free boundary velocity is greater than the characteristic velocity of such hyperbolic system, and, consequently, it is a space-like line \cite{d2019free, d2021free}. The initial-boundary conditions for the microbial concentrations are assigned on this line, and they are equal to the microbial species relative abundances in the biomass which attach on the granule-bulk liquid interface. The free boundary value problem is completed by two systems of semi linear elliptic partial differential equations that governs the quasi-static diffusion of substrates and planktonic species within the biofilm. Moreover, in Section \ref{n2} it is proved that equation describing the growth and transport of sessile biomass holds true for $r=0$. The complete integral version of the original differential free boundary problem is derived in Section \ref{n3} by introducing the method of characteristics. An existence and uniqueness theorem of solutions is shown in Section \ref{n4} in the class of continuous functions. Finally, the conclusions of the work are outlined in Section \ref{n5}.

\section{Modelling de novo granulation} \label{n2}
\
In this Section, the mathematical model introduced by Tenore et al. \cite{tenore2021multiscale} describing the genesis and growth process of a granular biofilm is discussed. The model has been conceived in the framework of continuum biofilm models \cite{wanner1986multispecies} and accounts for all the main processes involved in the biofilm lifecycle: microbial growth, substrates conversion, microbial invasion, attachment and  detachment.

\subsection{Equations} \label{n2.1}
\
Under the assumption of radial symmetry, the granular biofilm is modelled as a spherical free boundary domain. The granular biofilm center is located at $r=0$, where $r$ denotes the radial coordinate. The granular biofilm model takes into account the dynamics of three components, expressed in terms of concentration: $n$ microbial species in sessile form $X_i(r,t)$; $n$ microbial species in planktonic form $\Psi_i(r,t)$; and $m$ dissolved substrates $S_j(r,t)$. All variables involved in the biofilm modelling are considered as functions of time $t$ and space $r$. The density of the microbial species is denoted by $\rho_i$. By dividing the sessile species concentration $X_i(r,t)$ by $\rho_i$, biomass volume fractions $f_i(r,t)$ are achieved. Furthermore, $f_i$ are constrained to add up to unity at each location and time: $\sum_{i=1}^n f_i = 1$ \cite{rahman2015mixed}. In summary, the granular biofilm model components are:
\begin{equation} \label{2.1}
X_i, \ i=1,...,n, \ {\bf X}=(X_1,...,X_n),
\end{equation}

\begin{equation} \label{2.2}
f_i=\frac{X_i}{\rho_i}, \ i=1,...,n, \ {\bf f}=(f_1,...,f_n),
\end{equation}

\begin{equation} \label{2.3}
\Psi_i, \ i=1,...,n, \ \mbox{\boldmath $\Psi$}=(\Psi_1,...,\Psi_n),
\end{equation}

\begin{equation} \label{2.4}
S_j, \ j=1,...,m, \ {\bf S}=(S_1,...,S_m).
\end{equation}

Planktonic cells are supposed to have a negligible spatial occupancy due to the small particle size.

 The growth and transport of the $i^{th}$ sessile species across the granular biofilm is governed by the following system of non linear hyperbolic partial differential equations:
\begin{equation}
\begin{aligned}                                        \label{2.5}
 \frac{\partial X_i(r,t)}{\partial t} +\frac{1}{r^2} \frac{\partial}{\partial r}(r^2 u(r,t) X_i(r,t))
			 = \rho_i r_{M,i} + \rho_i r_i,  \ i=1,...,n, \ 0 \leq r \leq R(t), \ t>0,
\end{aligned}			 
\end{equation}
where $r_{M,i}$ and $r_{i}$ are the specific growth rates due to sessile and planktonic species, respectively; and $u(r,t)$ is the biomass velocity.

The function $r_{M,i}$ depends on sessile species $X_i$ and substrates $S_j$, while the function $r_{i}$ depends on planktonic species $\Psi_i$ and substrates $S_j$:
\begin{equation}                                        \label{2.7}
r_{M,i} = r_{M,i}({\bf X}(r,t), {\bf S}(r,t)),
\end{equation}

\begin{equation}                                        \label{2.8}
r_{i} = r_{i}(\mbox{\boldmath $\Psi$}(r,t), {\bf S}(r,t)).
\end{equation}

$u(r,t)$ is governed by the following equation:
\begin{equation}                                        \label{2.6}	
\frac{\partial u(r,t)}{\partial r} = -\frac{2 u(r,t)}{r} + G({\bf X}(r,t),{\bf S}(r,t),\mbox{\boldmath $\Psi$}(r,t)) , \ 0 < r \leq R(t), \ t>0,
\end{equation}
where $G({\bf X}(r,t),{\bf S}(r,t),\mbox{\boldmath $\Psi$}(r,t))= \sum_{i=1}^{n}(r_{M,i}+r_i)$.

Eq. \eqref{2.5} can be rewritten in terms of volume fractions $f_i$. By diving Eq. \eqref{2.5} by $\rho_i$ yields: 
 
\begin{equation}                                        \label{2.9}
\begin{aligned}
\frac{\partial f_i(r,t)}{\partial t} +\frac{1}{r^2} \frac{\partial}{\partial r}(r^2 u(r,t) f_i(r,t))=r_{M,i}+r_i, \ i=1,...,n, \ 0 \leq r \leq R(t), \ t>0.
\end{aligned}
\end{equation}	

The diffusion and conversion of soluble substrates and planktonic cells within the granular biofilm are governed by semi linear parabolic partial differential equations that are usually considered in quasi-static conditions \cite{d2019free, d2021free}:
\begin{equation}                                        \label{2.10}
\begin{aligned}
-D_{S,j}&\frac{1}{r^2}\frac{\partial }{\partial r}\bigg(r^2 \frac{\partial S_j(r,t)}{\partial r}\bigg)= r_{S,j}({{\bf X}(r,t)},{{\bf S}(r,t)}), \ j=1,...,m,  \ 0 < r < R(t), \ t>0,
\end{aligned}
\end{equation}

\begin{equation}                                        \label{2.11}
\begin{aligned}
-D_{\Psi,i}&\frac{1}{r^2}\frac{\partial }{\partial r}\bigg(r^2 \frac{\partial \Psi_i(r,t)}{\partial r}\bigg)=
r_{\Psi,i}({\mbox{\boldmath $\Psi$}(r,t)},{{\bf S}}(r,t)), \ i=1,...,n,  \ 0 < r < R(t), \ t>0,
\end{aligned}
 \end{equation}
where $r_{S,j}$ denotes the conversion rate of the $j^{th}$ dissolved substrate; $r_{\Psi,j}$ indicates the conversion rate due to the switch of the mode of growth from planktonic to sessile; $D_{S,i}$ and $D_{\Psi,i}$ are the diffusivity coefficients of the substrates and the planktonic species within the biofilm, respectively.

The biofilm radius $R(t)$ represents the free boundary of the mathematical problem. Its evolution depends on processes of microbial growth and phenomena which occurs on the surface of the granule. Attachment phenomena dominate the initial phase of the granulation process, while detachment phenomena become predominant as the granule size increases. The variation of $R(t)$ is governed by the following ordinary differential equation \cite{d2019free, d2021free, tenore2021multiscale}:
\begin{equation}                                       \label{2.12}
\dot R(t)
    = u(R(t),t) + \sigma_a(t)-\sigma_d(t),
\end{equation}
where $\sigma_a$ denotes the attachment velocity of biomass from bulk liquid to biofilm and $\sigma_d$ denotes the detachment velocity of biomass from biofilm to bulk liquid. The function $\sigma_a$ linearly depends on the concentrations of the microbial species in planktonic form $\Psi^*_i$, $i = 1, ..., n$, $\mbox{\boldmath $\Psi$}^* = (\Psi^*_1, ..., \Psi^*_n)$, suspended in the bulk liquid \cite{d2019free, d2021free,wanner1986multispecies, mavsic2014modeling}, each of which is characterized by a specific attachment velocity $v_{a,i}$:
\begin{equation}
\label{4.2.9}
    \sigma_a(t)=\sum_{i=1}^{n} \sigma_{a,i}(t)=\frac{\sum_{i=1}^{n}v_{a,i}\Psi^*_i(t)}{\rho_i},
\end{equation}
where attachment velocities $v_{a,i}$ can be assumed as constants or can be considered as functions of the environmental conditions affecting biofilm growth (i.e. substrates concentrations, biofilm composition, electrostatic and mechanical properties of the surface, etc.)

Meanwhile, the function $\sigma_d$ is modelled as a quadratic function of the granule radius $R(t)$ by following the same approach adopted for the planar case \cite{abbas2012longtime}: 
\begin{equation}                                        \label{2.13}
\sigma_d(t)=\delta R^2(t),
\end{equation}  
where $\delta$ is the detachment coefficient and is supposed to be equal for all microbial species. 

In the initial phase of biofilm formation, the attachment process prevails and the detachment process is very small, since so is $R^2$. Therefore, as shown by D'Acunto et al. \cite{d2021free} in this circumstances it is $\sigma_a - \sigma_d > 0$ and the free boundary velocity is greater than the characteristic velocity ($\dot R(t)>u(R,t)$). Thus, the spherical free boundary is a space-like line. For mature biofilms the spherical free boundary $R$ becomes greater, the detachment is the prevailing process ($\sigma_a - \sigma_d < 0$) and the free boundary is a time-like line.

\subsection{Initial-boundary conditions} \label{n2.2}
\
The following initial-boundary conditions are associated to the systems of PDEs \eqref{2.5}, \eqref{2.6}, \eqref{2.10} - \eqref{2.12}:
\begin{equation}                                        \label{2.2.5}
X_i(R(t),t)=X_{i,0}(t), \ i=1,...,n, \ t>0,
\end{equation}

\begin{equation}                                        \label{2.2.6}
u(0,t)=0, \ t>0,
\end{equation}

\begin{equation}                                        \label{2.2.10}
\frac{\partial S_j}{\partial r}(0,t)=0,\ S_j(R(t),t)=S^*_j(t),\ j=1,...,m, \ t>0,
\end{equation}

\begin{equation}                                        \label{2.2.11}
\frac{\partial \Psi_i}{\partial r}(0,t)=0,\ \Psi_i(R(t),t))=\Psi^*_i(t),\ i=1,...,n, \ t>0,
\end{equation}

\begin{equation}                                        \label{2.2.12}
R(0)=0.
\end{equation}
 
In Eq. \eqref{2.2.5}, $X_{i,0}(t)$ is the relative abundance of the $i^{th}$ species in the biomass attached to the granule-bulk liquid interface \cite{d2021free}. According to D'Acunto et al. \cite{d2021free}, $X_{i,0}(t)$ can be evaluated as follows:

\begin{equation}                                        \label{2.2.13}  
X_{i,0}(t) = \frac{v_{a,i}\Psi^*_i(t)}{\sum_{i=1}^{n}v_{a,i}\Psi^*_i(t)} \ \rho_i, \ i=1,...,n, \ t>0.
\end{equation} 

It means that the microbial species concentration at the granule-bulk liquid interface $X_i(R(t),t)$, for a multispecies granular biofilm under attachment regime, depends on the concentrations of the species in planktonic form present in the bulk liquid and its attachment properties. Note that, when all microbial species in the bulk liquid are characterized by the same attachment velocity, Eq. \eqref{2.2.13} reduces to:

\begin{equation}                                        \label{2.2.14}  
\frac{X_{i,0}(t)}{\rho_i} = \frac{\Psi^*_i(t)}{\sum_{i=1}^{n}\Psi^*_i(t)}, \ i=1,...,n, \ t>0,
\end{equation}
that is the volume fraction of the $i^{th}$ microbial species at the granule-bulk liquid interface assumes the same value of the volume fraction within the bulk liquid. As mentioned in Section \ref{n2.1}, in the initial phase of the granulation process, the free boundary is a space-like line ($\sigma_a>\sigma_d$), and the condition \eqref{2.2.5} is required. When the granule grows over time, the free boundary becomes a time-like line ($\sigma_a<\sigma_d$) and such condition is not required. The boundary condition \eqref{2.2.6} indicates that there is no biomass flux at the granule center ($r = 0$). Regarding the diffusion of substrates and planktonic species, the boundary conditions \eqref{2.2.10}$_1$ and \eqref{2.2.11}$_1$ are the no flux conditions at the granule center ($r = 0$). While, the boundary conditions \eqref{2.2.10}$_2$ and \eqref{2.2.11}$_2$ are Dirichlet conditions, which state that the values of the substrates and planktonic species on the free boundary are the same as in the bulk liquid. The functions $S^*_j(t)$ and $\Psi^*_j(t)$ are prescribed functions in general. To model the genesis of the granular biofilm (\textit{de novo} granulation) a vanishing initial domain is considered. For this reason, the initial condition \eqref{2.2.12} is coupled to Eq. \eqref{2.12}. It reproduces the case in which whereby only planktonic biomass is supposed to be present in the system at $t=0$. Note that Eq. \eqref{2.5} refers to the biofilm domain and does not require an initial condition, as the extension of the biofilm domain is zero at $t = 0$.

\subsection{Modelling the initial phase of multispecies granular biofilm growth} \label{n2.3}
\
In summary, the initial phase of multispecies granular biofilm growth is governed by the following spherical free boundary problem:

\begin{equation}              \label{2.3.1}
\begin{aligned}
& \frac{\partial X_i(r,t)}{\partial t} +\frac{1}{r^2} \frac{\partial}{\partial r}(r^2 u(r,t) X_i(r,t))
			 = \rho_i r_{M,i}({\bf X}(r,t), {\bf S}(r,t)) + \rho_i r_i(\mbox{\boldmath $\Psi$}(r,t), {\bf S}(r,t)), \\
& \ i=1,...,n, \ 0 \leq r \leq R(t), \ t>0,
\end{aligned}
\end{equation}	

\begin{equation}              \label{2.3.2}
X_i(R(t),t)=X_{i,0}(t), \ i=1,...,n, \ t>0,
\end{equation}

\begin{equation}              \label{2.3.3}
\dot R(t) = u(R(t),t) + \sigma_a(\mbox{\boldmath $\Psi$}^*), \ \sigma_a>0, \ t>0,
\end{equation}

\begin{equation}              \label{2.3.4}
R(0)=0,
\end{equation}

\begin{equation}              \label{2.3.5}		
\frac{\partial u(r,t)}{\partial r} = -\frac{2 u(r,t)}{r} +G({\bf X}(r,t),{\bf S}(r,t),\mbox{\boldmath $\Psi$}(r,t)) , \ 0 < r \leq R(t), \ t>0,
\end{equation}

\begin{equation}              \label{2.3.6}		
 u(0,t)=0,
\end{equation}

\begin{equation}              \label{2.3.7}
\begin{aligned}
-D_{S,j}\frac{1}{r^2}\frac{\partial }{\partial r}\bigg(r^2 \frac{\partial S_j(r,t)}{\partial r}\bigg)=
r_{S,j}({{\bf X}(r,t)},{{\bf S}(r,t)}), \ 0 < r < R(t),\ t>0, \ j=1,...,m,
\end{aligned}
\end{equation}

\begin{equation}              \label{2.3.8}
\frac{\partial S_j}{\partial r}(0,t)=0,\ S_j(R(t),t)=S^*_j(t),\ t>0, \ j=1,...,m,
\end{equation}

\begin{equation}              \label{2.3.9}
\begin{aligned}-D_{\Psi,i}\frac{1}{r^2}\frac{\partial }{\partial r}\bigg(r^2 \frac{\partial \Psi_i(r,t)}{\partial r}\bigg)=
r_{\Psi,i}({\mbox{\boldmath $\Psi$}(r,t)},{{\bf S}}(r,t)), \ 0 < r < R(t), \ t>0, \ \ i=1,...,n,
\end{aligned}
 \end{equation}

\begin{equation}               \label{2.3.10}
\frac{\partial \Psi_i}{\partial r}(0,t)=0,\ \Psi_i(R(t),t)=\Psi^*_i(t), \ t>0, \ i=1,...,n.
\end{equation}
 
Note that Eq. \eqref{2.3.3} refers to the initial phase of biofilm formation, when the detachment flux $\sigma_d$ is negligible compared to $\sigma_a$. The spherical free boundary $R(t)$ is a space-like line and Eq. \eqref{2.3.2} provides the initial conditions for the microbial species in sessile form on the free boundary.

\begin{rmk}\label{rmk}
 Eq. \eqref{2.3.1} has an apparent singularity for $r=0$. In this Section, it is proved that Eq. \eqref{2.3.1} holds also for $r=0$, as the singularity may be eliminated.
 
 Eq. \eqref{2.3.1} can be rewritten as:
\begin{equation}       \label{2.3.1.1}
\begin{aligned}
\frac{\partial X_i(r,t)}{\partial t} + X_i(r,t) \frac{\partial u(r,t)}{\partial r} + \frac{2 u(r,t) X_i(r,t)}{r} + u(r,t) \frac{\partial X_i(r,t)}{\partial r}= \rho_i r_{M,i} + \rho_i r_i, \ i=1,...,n.
\end{aligned}
\end{equation}
 
Consider the Taylor's series expansion of $u(r,t)$ for $r=0$:
\begin{equation}        \label{2.3.1.2}
u(r,t)= u(0,t)+\frac{\partial u(0,t)}{\partial r} r + ...
\end{equation}
 
Taking into account the boundary condition \eqref{2.3.6}, it follows:
\begin{equation}        \label{2.3.1.3}
\lim_{r\to 0} \frac{u(r,t)}{r} = \frac{\partial u(0,t)}{\partial r}.
\end{equation}
 
By considering Eq. \eqref{2.3.1.3}, Eqs. \eqref{2.3.5} and \eqref{2.3.1.1} for $r=0$ may be replaced by:
 
\begin{equation}        \label{2.3.1.4}		
\frac{\partial u(0,t)}{\partial r} = \frac{G({\bf X}(0,t),{\bf S}(0,t),\mbox{\boldmath $\Psi$}(0,t))}{3},
\end{equation}
	
\begin{equation}        \label{2.3.1.5} 
			\frac{\partial X_i(0,t)}{\partial t} + 3 X_i(0,t) \frac{\partial u(0,t)}{\partial r}=\rho_i r_{M,i} + \rho_i r_i.
\end{equation}	

Substituting Eq. \eqref{2.3.1.4} in Eq. \eqref{2.3.1.5} yields:

\begin{equation}        \label{2.3.1.6} 
\frac{\partial X_i(0,t)}{\partial t}
			 =\rho_i r_{M,i} + \rho_i r_i - X_i(0,t)G({\bf X}(0,t),{\bf S}(0,t),\mbox{\boldmath $\Psi$}(0,t)).
\end{equation}	

Eq. \eqref{2.3.1.6} replaces Eq. \eqref{2.3.1.1} for $r=0$ and confirms that Eq. \eqref{2.3.1} holds also for $r=0$.

\end{rmk}

\section{Integral system} \label{n3}
\
The system of equations \eqref{2.3.1}, \eqref{2.3.3}, \eqref{2.3.5}, \eqref{2.3.7}, and \eqref{2.3.9}, introduced in the previous section are here converted to an integral system by using the method of characteristics \cite{d2019free, d2021free,coclite1}. Consider the characteristic-like lines $r = r(t)$ of system \eqref{2.3.1}, defined by the differential equation:
\begin{equation}      \label{3.1} 
\frac{\partial r(t)}{\partial t}= u(r(t),t).
\end{equation}

Since they also depend on the starting point $t_0$, the notation $r(t) = c(t_0,t)$ can be used. Therefore, more precisely, the characteristics are defined by the following initial value problem:
\begin{equation}      \label{3.2}
\frac{\partial c}{\partial t}(t_0,t)
    =u(c(t_0,t),t),\ \ c(t_0,t_0)=R(t_0).
\end{equation}

In particular, for $t_0 = 0$ yields:
\begin{equation}      \label{3.3}
\frac{\partial c}{\partial t}(0,t)
    =u(c(0,t),t),\ \ c(0,0)=0,
\end{equation}
since $R(0) = 0$. The initial value problem \eqref{3.3} admits the solution $c(0, t) = 0$ because of condition \eqref{2.3.6}.

In characteristic coordinates, the integral form of Eq. \eqref{2.3.5} can be written as:
\begin{equation}     \label{3.4}
\begin{aligned}
u(c(t_0,t),t)=\frac{1}{c^2(t_0,t)} \int_{0}^{t_0} c^2(\tau,t) G({\bf X}(c(&\tau,t),t),{\bf S}(c(\tau,t),t),\mbox{\boldmath $\Psi$}(c(\tau,t),t)) \frac{\partial }{\partial \tau} c(\tau,t) d\tau.
\end{aligned}
\end{equation}

Consider Eq. \eqref{2.3.3} written for $t = t_0$: 
\begin{equation}      \label{3.5}
\dot R(t_0)= \sigma_a(\mbox{\boldmath $\Psi$}^*(t_0))+ u(R(t_0),t_0).
\end{equation}

Substituting Eq. \eqref{3.4} in Eq. \eqref{3.5} yields:
\begin{equation}       \label{3.6}
\begin{aligned}
\dot R(t_0)&= \sigma_a(\mbox{\boldmath $\Psi$}^*(t_0))\\
&+ \frac{1}{c^2(t_0,t_0)}\int_{0}^{t_0} c^2(\tau,t_0) G({\bf X}(c(\tau,t_0),t_0),{\bf S}(c(\tau,t_0),t_0),\mbox{\boldmath $\Psi$}(c(\tau,t_0),t_0)) \frac{\partial }{\partial \tau} c(\tau,t_0) d\tau,
\end{aligned}
 \end{equation}
since $u(R(t_0),t_0) = u(c(t_0,t_0),t_0)$.

After, integrating over $(0,t_0)$ it follows:
\begin{equation}       \label{3.7}
\begin{aligned}
R(t_0) &= \int_{0}^{t_0} \sigma_a(\mbox{\boldmath $\Psi$}^*(\theta)) d\theta \\
&+ \int_{0}^{t_0} \frac{1}{c^2(\theta,\theta)} d\theta \int_{0}^{\theta} c^2(\tau,\theta) G({\bf X}(c(\tau,\theta),\theta),{\bf S}(c(\tau,\theta),\theta),\mbox{\boldmath $\Psi$}(c(\tau,\theta),\theta))\frac{\partial }{\partial \tau} c(\tau,\theta) d\tau,
\end{aligned}
\end{equation}
where the condition \eqref{2.3.4} has been used. Equation \eqref{3.7} is the desired integral equation for $R$ in characteristic coordinates. 

Substituting Eq. \eqref{3.4} in Eq. \eqref{3.2}, it follows:
\begin{equation}        \label{3.8}
\begin{aligned}
\frac{\partial }{\partial t} c(t_0,t)=\frac{1}{c^2(t_0,t)}\int_{0}^{t_0} c^2(\tau,t) G({\bf X}(c(\tau,t),t),{\bf S}(c(\tau,t),t),\mbox{\boldmath $\Psi$}(c(\tau,t),t)) \frac{\partial c}{\partial \tau} (\tau,t) d\tau.
\end{aligned}
\end{equation}
 
Integrating over $(t_0,t)$:
\begin{equation}         \label{3.9}
\begin{aligned}
c(t_0,t)&- c(t_0,t_0) \\
&= \int_{t_0}^{t} \frac{1}{c^2(t_0,\theta)}d\theta \int_{0}^{t_0} c^2(\tau,\theta) G({\bf X}(c(\tau,\theta),\theta),{\bf S}(c(\tau,\theta),\theta),\mbox{\boldmath $\Psi$}(c(\tau,\theta),\theta)) \frac{\partial }{\partial \tau} c(\tau,\theta) d\tau,
\end{aligned}
\end{equation}
\begin{equation}         \label{3.10}
\begin{aligned}
 c(t_0,t) &= R(t_0)\\
& + \int_{t_0}^{t} \frac{1}{c^2(t_0,\theta)} d\theta \int_{0}^{t_0} c^2(\tau,\theta) G({\bf X}(c(\tau,\theta),\theta),{\bf S}(c(\tau,\theta),\theta),\mbox{\boldmath $\Psi$}(c(\tau,\theta),\theta)) \frac{\partial }{\partial \tau} c(\tau,\theta) d\tau.
 \end{aligned}
\end{equation}

Lastly, substituting expression \eqref{3.7} of $R(t_0)$ in Eq. \eqref{3.10} yields:
 \begin{equation}         \label{3.11}
\begin{aligned}
c(t_0,t) &= \int_{0}^{t_0} \sigma_a(\mbox{\boldmath $\Psi$}^*(\theta))d\theta \\
&+ \int_{0}^{t_0} \frac{1}{c^2(\theta,\theta)} d\theta \int_{0}^{\theta} c^2(\tau,\theta) G({\bf X}(c(\tau,\theta),\theta),{\bf S}(c(\tau,\theta),\theta),\mbox{\boldmath $\Psi$}(c(\tau,\theta),\theta)) \frac{\partial c}{\partial \tau} (\tau,\theta) d\tau \\
&+ \int_{t_0}^{t} \frac{1}{c^2(t_0,\theta)} d\theta \int_{0}^{t_0} c^2(\tau,\theta)  G({\bf X}(c(\tau,\theta),\theta),{\bf S}(c(\tau,\theta),\theta),\mbox{\boldmath $\Psi$}(c(\tau,\theta),\theta)) \frac{\partial c}{\partial \tau} (\tau,\theta) d\tau.
\end{aligned}
\end{equation}
   
Equation \eqref{3.11} is the desired integral equation for the characteristics. 
The integral equation for $R(t_0)$ in characteristic coordinates (Eq. \eqref{3.7}) can also be evaluated by considering Eq. \eqref{3.11} for $t=t_0$. 

In addition, the integral equation for $\partial c$/$\partial t_0$ can be obtained from Eq. \eqref{3.11} as follows: 

\begin{equation}         \label{3.12}
\begin{aligned}
\frac{\partial }{\partial t_0} c(t_0,t) = \sigma_a(\mbox{\boldmath $\Psi$}^*(t_0))
+ \int_{t_0}^{t} G({\bf X}(c(t_0,\theta),\theta),{\bf S}(c(t_0,\theta),\theta),\mbox{\boldmath $\Psi$}(c(t_0,\theta),\theta)) \frac{\partial }{\partial t_0} c(t_0,\theta) d\theta.
\end{aligned}
 \end{equation}
 
Eq. \eqref{3.12} is needed, as $\partial c$/$\partial t_0$ appears in the integral equations \eqref{3.7} - \eqref{3.11}.

Consider Eq. \eqref{2.3.1} written in characteristic coordinates:
\begin{equation}        \label{3.13}
\begin{aligned}
\frac{dX_i}{dt}(c(t_0,t),t)&= \rho_ir_{M,i}({\bf X}(c(t_0,t),t),{\bf S}(c(t_0,t),t))+\rho_ir_{i}(\mbox{\boldmath $\Psi$}(c(t_0,t),t),{\bf S}(c(t_0,t),t))\\
&-X_i(c(t_0,t),t)G({\bf X}(c(t_0,t),t),{\bf S}(c(t_0,t),t),\mbox{\boldmath $\Psi$}(c(t_0,t),t)), \ i = 1,..., n,
\end{aligned}
\end{equation}
where Eq. \eqref{2.3.5} has been used. The initial conditions for $t = t_0$ are derived from \eqref{2.3.2}:
\begin{equation}        \label{3.14}
X_i(c(t_0,t_0),t_0) = X_i(R(t_0),t_0) = X_{i,0}(t_0), \ i = 1,..., n.
\end{equation}

From \eqref{3.13}, it follows:
\begin{equation}        \label{3.15}
\begin{aligned}
\frac{dX_i}{dt}(c(t_0,t),t) = F_i({\bf X}(c(t_0,t),t),{\bf S}(c(t_0,t),t),\mbox{\boldmath $\Psi$}(c(t_0,t),t)), \ i = 1,..., n,
\end{aligned}
\end{equation}
where:
\begin{equation}        \label{3.16}
F_i=\rho_ir_{M,i}+\rho_ir_{i}-X_iG, \ i = 1,..., n.
\end{equation}

Integrating Eq. \eqref{3.15} over $(t_0,t)$ yields:
\begin{equation}        \label{3.17}
\begin{aligned}
& X_i(c(t_0,t),t)=X_{i,0}(t_0)
+\int_{t_0}^t F_i({\bf X}(c(t_0,\tau),\tau),{\bf S}(c(t_0,\tau),\tau),\mbox{\boldmath $\Psi$}(c(t_0,\tau),\tau))d\tau,\\
& \ 0\leq t_0< t \leq T, \ i=1,...,n,
\end{aligned}
\end{equation}
where the condition \eqref{3.14} has been used. The equation above is the desired integral equation for $X_i$ in characteristic coordinates. 

Consider Eq. \eqref{2.3.7} rewritten in characteristic coordinates:
\begin{equation}        \label{3.18}
\begin{aligned}
D_{S,j}\frac{1}{c^2(t_0,t)}\frac{\partial}{\partial r}\bigg(c^2(t_0,t)\frac{\partial S_j}{\partial r}(c(t_0,t),t)\bigg) =-r_{S,j}({{\bf X}(c(t_0,t),t)},{{\bf S}}(c(t_0,t),t)),\ j=1,...,m.
\end{aligned}
\end{equation}

The boundary conditions \eqref{2.3.8} for \eqref{3.18} assume the following expressions in characteristic coordinates:
\begin{equation}        \label{3.19}
\frac{\partial S_j}{\partial r}(0,t)=\frac{\partial S_j}{\partial r}(c(0,t),t)=0, \ S_j(R(t),t)=S_j(c(t,t),t))=S^*_j(t),
\end{equation}
because of \eqref{3.2} and \eqref{3.3}. From \eqref{3.18}:
\begin{equation}        \label{3.20}
\begin{aligned}
D_{S,j}\frac{\partial}{\partial t_0}\bigg(c^2(t_0,t)\frac{\partial S_j}{\partial r}(c(t_0,t),t)\bigg)
=-c^2(t_0,t) r_{S,j}({{\bf X}(c(t_0,t),t)},{{\bf S}}(c(t_0,t),t))\frac{\partial }{\partial t_0} c(t_0,t).
\end{aligned}
\end{equation}

Integrating the equation above over $(0, t_0)$ yields:
\begin{equation}         \label{3.21}
\begin{aligned}
D_{S,j}c^2(t_0,t)\frac{\partial S_j}{\partial r}(c(t_0,t),t)=-\int_{0}^{t_0} c^2(\tau,t) r_{S,j}({{\bf X}(c(\tau,t),t)},{{\bf S}}(c(\tau,t),t))\frac{\partial }{\partial \tau} c(\tau,t)d\tau,
\end{aligned}
\end{equation}
where the boundary condition \eqref{3.19}$_1$ has been used. From \eqref{3.21}:
\begin{equation}         \label{3.22}
\begin{aligned}
&D_{S,j}\frac{\partial }{\partial t_0}S_j(c(t_0,t),t)
\\
&=-\frac{1}{c^2(t_0,t)}\frac{\partial }{\partial t_0} c(t_0,t)\int_0^{t_0} c^2(\tau,t)r_{S,j}({{\bf X}(c(\tau,t),t)},{{\bf S}}(c(\tau,t),t))\frac{\partial }{\partial \tau} c(\tau,t)d\tau.
\end{aligned}
\end{equation}

Integrating the equation above over $(t_0, t)$ yields:
\begin{equation}        \label{3.23}
\begin{aligned}
&D_{S,j}S_j(c(t_0,t),t) =D_{S,j}S^*_j(t)\\
&+\int_{t_0}^t \frac{1}{c^2(\theta,t)}\frac{\partial }{\partial \theta} c(\theta,t)d\theta\int_{0}^{\theta} c^2(\tau,t)r_{S,j}({\bf X}(c(\tau,t),t),{\bf S}(c(\tau,t),t))\frac{\partial }{\partial \tau} c(\tau,t)d\tau,\\
&\ 0\leq t_0<t \leq T, \ j=1,...,m,
\end{aligned}
\end{equation}
where the boundary condition \eqref{3.19}$_2$ has been used. The equation above is the desired integral equation for $S_j$ in characteristic coordinates. 

Similarly, consider Eq. \eqref{2.3.9} rewritten in characteristic coordinates:
\begin{equation}        \label{3.24}
\begin{aligned}
D_{\Psi,j}\frac{1}{c^2(t_0,t)}\frac{\partial}{\partial r}\bigg(c^2(t_0,t)\frac{\partial \Psi_i}{\partial r}(c(t_0,t),t)\bigg) =- r_{\Psi,i}({\mbox{\boldmath $\Psi$}(c(t_0,t),t)},{{\bf S}}(c(t_0,t),t)),\ i=1,...,n.
\end{aligned}
\end{equation}

The boundary conditions \eqref{2.3.10} for \eqref{3.24} assume the following expression in characteristic coordinates:
\begin{equation}        \label{3.25}
\frac{\partial \Psi_i}{\partial r}(0,t)=\frac{\partial \Psi_i}{\partial r}(c(0,t),t)=0, \ \Psi_i(R(t),t)=\Psi_i(c(t,t),t))=\Psi^*_i(t),
\end{equation}
because of \eqref{3.2} and \eqref{3.3}. From \eqref{3.24}:
\begin{equation}        \label{3.26}
\begin{aligned}
D_{\Psi,j}\frac{\partial}{\partial t_0}\bigg(c^2(t_0,t)\frac{\partial \Psi_i}{\partial r}(c(t_0,t),t)\bigg)
=-c^2(t_0,t) r_{\Psi,i} ({\mbox{\boldmath $\Psi$}(c(t_0,t),t)},{{\bf S}}(c(t_0,t),t))\frac{\partial }{\partial t_0} c(t_0,t).
\end{aligned}
\end{equation}

Integrating the equation above over $(0, t_0)$ yields:
\begin{equation}        \label{3.27}
\begin{aligned}
D_{\Psi,i}c^2(t_0,t)\frac{\partial \Psi_i}{\partial r}(c(t_0,t),t) 
=-\int_{0}^{t_0} c^2(\tau,t) r_{\Psi,i}({\mbox{\boldmath $\Psi$}(c(\tau,t),t)},{{\bf S}}(c(\tau,t),t))\frac{\partial }{\partial \tau} c(\tau,t)d\tau,
\end{aligned}
\end{equation}
where the boundary condition \eqref{3.25}$_1$ has been used. From \eqref{3.27}:
\begin{equation}        \label{3.28}
\begin{aligned}
&D_{\Psi,i}\frac{\partial }{\partial t_0}\Psi_i(c(t_0,t),t) 
\\
&=-\frac{1}{c^2(t_0,t)} \frac{\partial }{\partial t_0} c(t_0,t)\int_0^{t_0} c^2(\tau,t) r_{\Psi,i}({\mbox{\boldmath $\Psi$}(c(\tau,t),t)},{{\bf S}}(c(\tau,t),t))\frac{\partial }{\partial \tau} c(\tau,t)d\tau.
\end{aligned}
\end{equation}

Integrating the equation above over $(t_0, t)$ yields:
\begin{equation}         \label{3.29}
\begin{aligned}
&D_{\Psi,i}\Psi_i(c(t_0,t),t) =D_{\Psi,i}\Psi^*_i(t)
\\
&+\int_{t_0}^t \frac{1}{c^2(\theta,t)} \frac{\partial }{\partial \theta} c(\theta,t)d\theta\int_{0}^{\theta} c^2(\tau,t) r_{\Psi,i}(\mbox{\boldmath $\Psi$}(c(\tau,t),t),{\bf S}(c(\tau,t),t))\frac{\partial }{\partial \tau} c(\tau,t)d\tau,\\
&\ 0\leq t_0<t \leq T, \ i=1,...,n,
\end{aligned}
\end{equation}
where the boundary condition \eqref{3.25}$_2$ has been used. The equation above is the desired integral equation for $\Psi_i$ in characteristic coordinates. 

Note that an integral equation for $R(t_0)$ and $c(t_0,t)$ can be obtained by considering the following reasonings as well. 

The integral form of Eq. \eqref{3.6} can be written as follows:
\begin{equation}       \label{3.6.1}
\begin{aligned}
R^2(t_0)\dot R(t_0)&= R^2(t_0)\sigma_a(\mbox{\boldmath $\Psi$}^*(t_0))\\
&+ \int_{0}^{t_0} c^2(\tau,t_0)G({\bf X}(c(\tau,t_0),t_0),{\bf S}(c(\tau,t_0),t_0),\mbox{\boldmath $\Psi$}(c(\tau,t_0),t_0)) \frac{\partial }{\partial \tau} c(\tau,t_0) d\tau,
\end{aligned}
 \end{equation}
where the condition \eqref{3.2}$_2$ has been used. 
\begin{equation}       \label{3.6.1.2}
\begin{aligned}
\dot R^3(t_0)&= 3R^2(t_0)\sigma_a(\mbox{\boldmath $\Psi$}^*(t_0))\\
&+ 3\int_{0}^{t_0} c^2(\tau,t_0) G({\bf X}(c(\tau,t_0),t_0),{\bf S}(c(\tau,t_0),t_0),\mbox{\boldmath $\Psi$}(c(\tau,t_0),t_0)) \frac{\partial }{\partial \tau} c(\tau,t_0) d\tau.
\end{aligned}
 \end{equation} 

After integrating over $(0,t_0)$, yields:
\begin{equation}       \label{3.7.1}
\begin{aligned}
R^3(t_0) &= 3\int_{0}^{t_0}R^2(\theta) \sigma_a(\mbox{\boldmath $\Psi$}^*(\theta)) d\theta \\
 &+3 \int_{0}^{t_0}  d\theta \int_{0}^{\theta} c^2(\tau,\theta) G({\bf X}(c(\tau,\theta),\theta),{\bf S}(c(\tau,\theta),\theta),\mbox{\boldmath $\Psi$}(c(\tau,\theta),\theta)) \frac{\partial }{\partial \tau} c(\tau,\theta) d\tau,
 \end{aligned}
\end{equation}

\begin{equation}       \label{3.7.1.2}
\begin{aligned}
 R(t_0) &= \bigg( 3\int_{0}^{t_0}R^2(\theta) \sigma_a(\mbox{\boldmath $\Psi$}^*(\theta)) d\theta \\
 &+3 \int_{0}^{t_0}  d\theta \int_{0}^{\theta} c^2(\tau,\theta) G({\bf X}(c(\tau,\theta),\theta),{\bf S}(c(\tau,\theta),\theta),\mbox{\boldmath $\Psi$}(c(\tau,\theta),\theta))  \frac{\partial }{\partial \tau} c(\tau,\theta) d\tau \bigg)^{\frac{1}{3}}.
\end{aligned}
\end{equation}

Equation \eqref{3.7.1.2} is the desired integral equation for $R$ in characteristic coordinates. 

Similarly, the integral form of Eq. \eqref{3.8} can be written as follows:
\begin{equation}        \label{3.8.1}
c^2(t_0,t)\frac{\partial }{\partial t} c(t_0,t)= \int_{0}^{t_0} c^2(\tau,t) G({\bf X}(c(\tau,t),t),{\bf S}(c(\tau,t),t),\mbox{\boldmath $\Psi$}(c(\tau,t),t))\frac{\partial c}{\partial \tau} (\tau,t)  d\tau,
\end{equation}

\begin{equation}        \label{3.8.1.2}
\frac{\partial }{\partial t} c^3(t_0,t)= 3 \int_{0}^{t_0} c^2(\tau,t) G({\bf X}(c(\tau,t),t),{\bf S}(c(\tau,t),t),\mbox{\boldmath $\Psi$}(c(\tau,t),t)) \frac{\partial c}{\partial \tau} (\tau,t) d\tau.
\end{equation}
 
Integrating over $(t_0,t)$:
\begin{equation}         \label{3.9.1}
\begin{aligned}
&c^3(t_0,t)- c^3(t_0,t_0) \\
&= 3\int_{t_0}^{t} d\theta \int_{0}^{t_0} c^2(\tau,\theta) G({\bf X}(c(\tau,\theta),\theta),{\bf S}(c(\tau,\theta),\theta),\mbox{\boldmath $\Psi$}(c(\tau,\theta),\theta)) \frac{\partial }{\partial \tau} c(\tau,\theta) d\tau,
\end{aligned}
\end{equation}
 
\begin{equation}         \label{3.10.1}
\begin{aligned}
c^3(t_0,t) &= R^3(t_0)\\
&+ 3\int_{t_0}^{t} d\theta \int_{0}^{t_0} c^2(\tau,\theta) G({\bf X}(c(\tau,\theta),\theta),{\bf S}(c(\tau,\theta),\theta),\mbox{\boldmath $\Psi$}(c(\tau,\theta),\theta)) \frac{\partial }{\partial \tau} c(\tau,\theta) d\tau.
 \end{aligned}
\end{equation}

Lastly, substituting expression \eqref{3.7.1} of $R^3(t_0)$ in Eq. \eqref{3.10.1} yields:
\begin{equation}         \label{3.11.1}
\begin{aligned}
c^3(t_0,t) &= 3\int_{0}^{t_0} R^2(\theta)\sigma_a(\mbox{\boldmath $\Psi$}^*(\theta))d\theta \\
&+3\int_{0}^{t_0} d\theta \int_{0}^{\theta} c^2(\tau,\theta) G({\bf X}(c(\tau,\theta),\theta),{\bf S}(c(\tau,\theta),\theta),\mbox{\boldmath $\Psi$}(c(\tau,\theta),\theta)) \frac{\partial c}{\partial \tau} (\tau,\theta) d\tau \\
&+3 \int_{t_0}^{t} d\theta \int_{0}^{t_0} c^2(\tau,\theta)  G({\bf X}(c(\tau,\theta),\theta),{\bf S}(c(\tau,\theta),\theta),\mbox{\boldmath $\Psi$}(c(\tau,\theta),\theta)) \frac{\partial c}{\partial \tau} (\tau,\theta) d\tau,
\end{aligned}
\end{equation}

\begin{equation}         \label{3.11.1.2}
\begin{aligned}
c(t_0,t) &= \bigg( 3\int_{0}^{t_0} R^2(\theta)\sigma_a(\mbox{\boldmath $\Psi$}^*(\theta))d\theta \\
&+ 3\int_{0}^{t_0} d\theta \int_{0}^{\theta} c^2(\tau,\theta) G({\bf X}(c(\tau,\theta),\theta),{\bf S}(c(\tau,\theta),\theta),\mbox{\boldmath $\Psi$}(c(\tau,\theta),\theta)) \frac{\partial c}{\partial \tau} (\tau,\theta) d\tau \\
&+3 \int_{t_0}^{t} d\theta \int_{0}^{t_0} c^2(\tau,\theta)  G({\bf X}(c(\tau,\theta),\theta),{\bf S}(c(\tau,\theta),\theta),\mbox{\boldmath $\Psi$}(c(\tau,\theta),\theta)) \frac{\partial c}{\partial \tau} (\tau,\theta) d\tau \bigg)^{\frac{1}{3}}.
\end{aligned}
\end{equation}
   
Equation \eqref{3.11.1.2} is the desired integral equation for the characteristics. 

In addition, the partial derivative of $c^3(t_0,t)$ with respect to $t_0$ satisfies the following integral equation:
\begin{equation}         \label{3.12.1}
\begin{aligned}
\frac{\partial }{\partial t_0} c^3(t_0,t) &= 3 R^2(t_0)\sigma_a(\mbox{\boldmath $\Psi$}^*(t_0))\\
&+3 \int_{t_0}^{t} c^2(t_0,\theta)G({\bf X}(c(t_0,\theta),\theta),{\bf S}(c(t_0,\theta),\theta),\mbox{\boldmath $\Psi$}(c(t_0,\theta),\theta)) \frac{\partial }{\partial t_0} c(t_0,\theta) d\theta,
\end{aligned}
 \end{equation}
 
 \begin{equation}         \label{3.12.1.2}
\begin{aligned}
\frac{\partial }{\partial t_0} c(t_0,t) &= \frac{1}{c^2(t_0,t)} R^2(t_0)\sigma_a(\mbox{\boldmath $\Psi$}^*(t_0))\\
&+ \frac{1}{c^2(t_0,t)}\int_{t_0}^{t} c^2(t_0,\theta)G({\bf X}(c(t_0,\theta),\theta),{\bf S}(c(t_0,\theta),\theta),\mbox{\boldmath $\Psi$}(c(t_0,\theta),\theta)) \frac{\partial }{\partial t_0} c(t_0,\theta) d\theta.
\end{aligned}
 \end{equation}
 
Eq. \eqref{3.12.1.2} is needed, as $\partial c$/$\partial t_0$ appears in the integral equations \eqref{3.7.1.2} - \eqref{3.11.1.2}.

\section{Uniqueness and existence theorem} \label{n4}
\
An existence and uniqueness result for the integral system \eqref{3.7}, \eqref{3.11}, \eqref{3.12}, \eqref{3.17}, \eqref{3.23} and \eqref{3.29} is derived in this section under the hypothesis $\sigma_d = 0$. Note that the free boundary coincides with the initial point of the characteristic line $r = c(t_0, t)$, whose evolution is governed by equation \eqref{3.11}.

By using the following positions:
\begin{equation}         \label{4.3.1}
{\bf x}(t_0,t)={\bf X}(c(t_0,t),t),\ \ {\bf x}(x_1,...,x_n),
\end{equation}
 
\begin{equation}         \label{4.3.2}
{\bf s}(t_0,t)={\bf S}(c(t_0,t),t),\ \ {\bf s}(s_1,...,s_m),
\end{equation}
 
\begin{equation}         \label{4.3.3}
\mbox{\boldmath $\psi$}(t_0,t)=\mbox{\boldmath $\Psi$}(c(t_0,t),t),\ \ \mbox{\boldmath $\psi$}(\psi_1,...,\psi_n),
\end{equation}
 
\begin{equation}         \label{4.3.4}
\begin{aligned}
&F_{c,1}({\bf x}(\tau,\theta),{\bf s}(\tau,\theta),\mbox{\boldmath $\psi$}(\tau,\theta),c(\theta,\theta),c(\tau,\theta),\frac{\partial c}{\partial \tau}(\tau,\theta))\\
&= \frac{1}{c^2(\theta,\theta)} c^2(\tau,\theta) G({\bf x}(\tau,\theta),{\bf s}(\tau,\theta),\mbox{\boldmath $\psi$}(\tau,\theta))\frac{\partial c}{\partial \tau}(\tau,\theta),
\end{aligned}
\end{equation}

\begin{equation}         \label{4.3.5}
\begin{aligned}
&F_{c,2}({\bf x}(\tau,\theta),{\bf s}(\tau,\theta),\mbox{\boldmath $\psi$}(\tau,\theta),c(t_0,\theta),c(\tau,\theta),\!\frac{\partial c}{\partial \tau}(\tau,\theta))\\
&=  \frac{1}{c^2(t_0,\theta)} c^2(\tau,\theta) G({\bf x}(\tau,\theta),{\bf s}(\tau,\theta),\mbox{\boldmath $\psi$}(\tau,\theta))\frac{\partial c}{\partial \tau}(\tau,\theta),
\end{aligned}
\end{equation}

\begin{equation}         \label{4.3.6}
\begin{aligned}
&F_{c,3}({\bf x}(t_0,\theta),{\bf s}(t_0,\theta),\mbox{\boldmath $\psi$}(t_0,\theta),
\frac{\partial c}{\partial t_0}(t_0,\theta))\\
&=G({\bf x}(t_0,\theta),{\bf s}(t_0,\theta),\mbox{\boldmath $\psi$}(t_0,\theta))
\frac{\partial c}{\partial t_0}(t_0,\theta),
\end{aligned}
\end{equation}

\begin{equation}         \label{4.3.7}                                  
\begin{aligned}
&F_{s,j}({\bf x}(\tau,t),{\bf s}(\tau,t),c(\theta,t),c(\tau,t),\!
\frac{\partial c}{\partial \theta}(\theta,t),
\frac{\partial c}{\partial \tau}(\tau,t))\! \\
=&D_{S,j}^{-1} \frac{1}{c^2(\theta,t)} c^2(\tau,t) r_{S,j}({\bf x}(\tau,t),{\bf s}(\tau,t))
\frac{\partial c}{\partial \theta}(\theta,t)
\frac{\partial c}{\partial \tau}(\tau,t),
\end{aligned}          
\end{equation}

\begin{equation}         \label{4.3.8}                            
\begin{aligned}
&F_{\psi,i}(\mbox{\boldmath $\psi$}(\tau,t),{\bf s}(\tau,t),c(\theta,t),c(\tau,t),\!
\frac{\partial c}{\partial \theta}(\theta,t),\!
\frac{\partial c}{\partial \tau}(\tau,t))
\! \\
=D&_{\psi,i}^{-1} \frac{1}{c^2(\theta,t)} c^2(\tau,t)
r_{\psi,i}(\mbox{\boldmath $\psi$}(\tau,t),{\bf s}(\tau,t))
\frac{\partial c}{\partial \theta}(\theta,t)
\frac{\partial c}{\partial \tau}(\tau,t),     
\end{aligned}
\end{equation}


\begin{equation}           \label{4.3.10}
\Sigma(t_0)=\int_{0}^{t_0}\sigma_{a}(\mbox{\boldmath $\Psi$}^*(\theta))
d\theta,
\end{equation}  
the integral system \eqref{3.11}, \eqref{3.12}, \eqref{3.17}, \eqref{3.23} and \eqref{3.29} is converted into the following more compact equations:
\begin{equation}           \label{4.3.11}
\begin{aligned}
c(t_0,t) &= \Sigma(t_0) \\
&+  \int_{0}^{t_0} d\theta \int_{0}^{\theta} F_{c,1}({\bf x}(\tau,\theta),{\bf s}(\tau,\theta),\mbox{\boldmath $\psi$}(\tau,\theta),c(\theta,\theta),c(\tau,\theta),
\frac{\partial c}{\partial \tau}(\tau,\theta)) d\tau \\
&+  \int_{t_0}^{t} d\theta \int_{0}^{t_0}  F_{c,2}({\bf x}(\tau,\theta),{\bf s}(\tau,\theta),\mbox{\boldmath $\psi$}(\tau,\theta),c(t_0,\theta),c(\tau,\theta),
\frac{\partial c}{\partial \tau}(\tau,\theta)) d\tau, \ 0< t_0< t\leq T,
\end{aligned}
\end{equation}

\begin{equation}            \label{4.3.12}
\begin{aligned}
\frac{\partial c}{\partial t_0}(t_0,t)
=\sigma_{a}(\mbox{\boldmath $\Psi$}^*(t_0))+\int_{t_0}^{t} F_{c,3}({\bf x}(t_0,\theta),{\bf s}(t_0,\theta),\mbox{\boldmath $\psi$}(t_0,\theta),
\frac{\partial c}{\partial t_0}(t_0,\theta))
d\theta,\ \ 0< t_0< t\leq T,
\end{aligned}
\end{equation}

\begin{equation}            \label{4.3.13}
\begin{aligned}
x_i(t_0,t)=X_{i,0}(t_0)+\int_{t_0}^t F_i({\bf x}(t_0,\tau),{\bf s}(t_0,\tau),\mbox{\boldmath $\psi$}(t_0,\tau))d\tau,
\ 0\leq t_0< t\leq T, \ i=1,...,n,
\end{aligned}
\end{equation}

\begin{equation}            \label{4.3.14}
\begin{aligned}
&s_j(t_0,t)= S_j^*(t)+ \int_{t_0}^{t}d\theta\int_{0}^{\theta}F_{s,j}({\bf x}(\tau,t),{\bf s}(\tau,t),c(\theta,t),c(\tau,t),
\frac{\partial c}{\partial \theta}(\theta,t),
\frac{\partial c}{\partial \tau}(\tau,t))d\tau,\\
& \ 0< t_0< t\leq T, \ j=1,...,m,
\end{aligned}
\end{equation}

\begin{equation}            \label{4.3.15}
\begin{aligned}
&\psi_i(t_0,t)=\Psi_i^*(t)
+ \int_{t_0}^{t}d\theta\int_{0}^{\theta}F_{\psi,i}(\mbox{\boldmath $\psi$}(\tau,t),{\bf s}(\tau,t),c(\theta,t),c(\tau,t),\frac{\partial c}{\partial \theta}(\theta,t),
\frac{\partial c}{\partial \tau}(\tau,t))d\tau,\\
&\ \ 0< t_0< t\leq T, \ i=1,...,n.
\end{aligned}
\end{equation}

 
An existence and uniqueness theorem for the integral problem \eqref{4.3.11}-\eqref{4.3.15} can be proved in the space of the continuous functions as generalization of the results in D'Acunto et al. \cite{d2019free, d2021free}.

\begin{thm}\label{thm}
Suppose that:

(a) $x_{i}^{}(t_0^{},t),s_{j}^{}(t_0^{},t),\psi_{i}^{}(t_0^{},t),c(t_0^{},t),c_{t_0}^{}(t_0^{},t)
 \in C^0([0,\ T_1]\times[0,\  T_1])$,
 \ $T_1>0$, \ $i=1,...,n$, \ $j=1,...,m$
 ;

 (b) $X_{i,0}(t_0^{}), \sigma_{a}(\mbox{\boldmath $\Psi^*$}(t_0)), S_j^*(t), \Psi_i^*(t)
 \in C^0([0,\ T_1]),
 \ i=1,...,n, \ j=1,...,m$;

 (c) $|x_{i}^{}-X_{i,0}|\leq h_{x,i}, \ i=1,...,n$;
 $|s_{j}^{}-S_j^*|\leq h_{s,j}, \ j=1,...,m$;
 $|\psi_{i}^{}-\Psi_i^*|\leq h_{\psi,i}, \ i=1,...,n$;
  $|c-\Sigma|\leq h_{c,1}$;
 $|c_{t_0}^{}-\sigma_a|\leq h_{c,2}$,
 where $h_{x,i},h_{s,j},h_{\psi,i},$
 $h_{c,1},h_{c,2}$ are positive constants;

 (d) $F_i, i=1,...,n$, $F_{s,j}, j=1,...,m$, $F_{\psi,i}, i=1,...,n$, 
 $F_{c,1},F_{c,2},F_{c,3}$ are bounded and Lipschitz continuous with respect to their arguments
\[
 M_i=\max |F_i|,\ i=1,...,n,\ M_{s,j}=\max |F_{s,j}|,\ j=1,...,m,
\]
\[
 M_{\psi,i}=\max |F_{\psi,i}|, \ i=1,...,n, \ M_{c,1}=\max(|F_{c,1}|,|F_{c,2}|), \ M_{c,2}=\max |F_{c,3}|,
\]
 
\[
 |F_i({\bf x},{\bf s},\mbox{\boldmath $\psi$})-F_i(\tilde{\bf x},\tilde{\bf s},\tilde{\mbox{\boldmath $\psi$}})|\leq
 \lambda_i\left[\sum_{k=1}^{n}|x_{k}^{}-\tilde x_{k}^{}|
 +\sum_{k=1}^{m}|s_{k}^{}-\tilde s_{k}^{}|+\sum_{k=1}^{n}|\psi_{k}^{}-\tilde \psi_{k}^{}|\right],
 \ i=1,...n,
 \]
 \[
 \begin{aligned}
 &|F_{s,j}({\bf x},{\bf s},c,c_{t_0}^{})-
 F_{s,j}(\tilde{\bf x},\tilde{\bf s},\tilde{c},\tilde c_{t_0}^{})|
 \leq 
 \lambda_{s,j}\left[\sum_{k=1}^{n}|x_{k}^{}-\tilde x_{k}^{}|
 +\sum_{k=1}^{m}|s_{k}^{}-\tilde s_{k}^{}|+|c-\tilde c|+|c_{t_0}^{}-\tilde c_{t_0}^{}| \right],\\
 & \ j=1,...m,
 \end{aligned}
 \]
\[
\begin{aligned} 
 &|F_{\psi,i}(\mbox{\boldmath $\psi$},{\bf s},c,c_{t_0}^{})-
 F_{\psi,i}(\tilde{\mbox{\boldmath $\psi$}},\tilde{\bf s},\tilde{c},\tilde c_{t_0}^{})|
 \leq
 \lambda_{\psi,i}\left[\sum_{k=1}^{n}|\psi_{k}^{}-\tilde \psi_{k}^{}|
 \!+\!\sum_{k=1}^{m}|s_{k}^{}-\tilde s_{k}^{}|\!+\!|c-\tilde c|+|c_{t_0}^{}-\tilde c_{t_0}^{}|
 \right], \\
 & \ i=1,...,n,
 \end{aligned}
 \]
 \[
 \begin{aligned}
 |F_{c,1}({\bf x},{\bf s},\mbox{\boldmath $\psi$},c,c_{t_0}^{})-F&_{c,1}(\tilde{\bf x},\tilde{\bf s},\tilde{\mbox{\boldmath $\psi$}},\tilde c,\tilde c_{t_0}^{})|\leq \\
 & \lambda_{c,1}\left[\!\sum_{k=1}^{n}\!|x_{k}^{}-\tilde x_{k}^{}|
 +\!\sum_{k=1}^{m}\!|s_{k}^{}-\tilde s_{k}^{}|+\!\sum_{k=1}^{n}\!|\psi_{k}^{}-\tilde \psi_{k}^{}|+|c-\tilde c|+|c_{t_0}^{}-\tilde c_{t_0}^{}|
 \right],
 \end{aligned}
 \]
 \[
 \begin{aligned}
 |F_{c,2}({\bf x},{\bf s},\mbox{\boldmath $\psi$},c,c_{t_0}^{})-F&_{c,2}(\tilde{\bf x},\tilde{\bf s},\tilde{\mbox{\boldmath $\psi$}},\tilde c,\tilde c_{t_0}^{})|\leq\\
& \lambda_{c,2}\left[\!\sum_{k=1}^{n}\!|x_{k}^{}-\tilde x_{k}^{}|
 \!+\sum_{k=1}^{m}\!|s_{k}^{}-\tilde s_{k}^{}|+\!\sum_{k=1}^{n}\!|\psi_{k}^{}-\tilde \psi_{k}^{}|+|c-\tilde c|+|c_{t_0}^{}-\tilde c_{t_0}^{}|
 \right],
\end{aligned} 
 \]
 \[
 \begin{aligned}
 |F_{c,3}({\bf x},{\bf s},\mbox{\boldmath $\psi$},c_{t_0}^{})-F_{c,2}(\tilde{\bf x},\tilde{\bf s},\tilde{\mbox{\boldmath $\psi$}}&,\tilde c_{t_0}^{})|\leq \\
& \lambda_{c,3}\left[\sum_{k=1}^{n}|x_{k}^{}-\tilde x_{k}^{}|
 +\sum_{k=1}^{m}|s_{k}^{}-\tilde s_{k}^{}|+\sum_{k=1}^{n}|\psi_{k}^{}-\tilde \psi_{k}^{}|+|c_{t_0}^{}-\tilde c_{t_0}^{}|
 \right],
\end{aligned}
 \]
 {\it when}
 $(t_0^{},t)\in[0,\  T_1]\times[0,\  T_1]$ and the functions
 $x_{i}^{}$, $s_{j}^{}$, $\psi_{i}^{}$, 
  $c$, $c_{t_0}^{}$ satisfy the assumptions (a)-(c).

 Then, integral system \eqref{4.3.11}-\eqref{4.3.15}
 has a unique solution $x_{i}^{}$, $s_{j}^{}$, $\psi_{i}^{}$, 
  $c$, $c_{t_0}^{}$,
 $\in C^0([0,\ T]\times[0,\  T])$,
 where
 \[
 T=\min\left\{T_1,\frac{h_{x,1}}{M_1},...,\frac{h_{x,n}}{M_n},
 \sqrt{\frac{h_{s,1}}{M_{s,1}}},...,\sqrt{\frac{h_{s,m}}{M_{s,m}}},
 \sqrt{\frac{h_{\psi,1}}{M_{\psi,1}}},...,\sqrt{\frac{h_{\psi,n}}{M_{\psi,n}}},  \sqrt{\frac{h_{c,1}}{2M_{c,1}}},
 \frac{h_{c,2}}{M_{c,2}}
 \right\}.
 \]

Moreover, $T$ satisfies the following condition,
 \begin{equation}                                        \label{4.3.16}
  aT^2+bT<1,
 \end{equation}
 where
 \begin{equation}                                        \label{4.3.17}
  a=\sum_{j=1}^{m}\lambda_{s,j}+\sum_{i=1}^{n}\lambda_{\psi,i}+\lambda_{c,1}+\lambda_{c,2},\
  b=\sum_{i=1}^{n}\lambda_{i}+\lambda_{c,3}.
 \end{equation}
\end{thm}

 \begin{proof}
 Denote by $\Omega$ the space of continuous functions
 $x_{i}(t_0^{},t)$, $s_{j}(t_0^{},t)$, $\psi_{i}(t_0^{},t)$, 
 $c(t_0^{},t)$,
 $c_{t_0}^{}(t_0^{},t)$,\ $t_0^{}\in[0,\ T]$,\ $t\in[0,\ T]$, and endow it with the uniform norm

\[
 ||({\bf x},{\bf s},\mbox{\boldmath $\psi$},c,c_{t_0}^{})||=
\sum_{i=1}^{n}\max_{\Omega}|x_i^{}|
 +\sum_{j=1}^{m}\max_{\Omega}|s_j^{}|+\sum_{i=1}^{n}\max_{\Omega}|\psi_i^{}|
  +
 \max_{\Omega}|c|+\max_{\Omega}|c_{t_0}^{}|.
 \]
 Consider the map
 $({\bf x}^*,{\bf s}^*,\mbox{\boldmath ${\psi}^*$},c^{*},c_{t_0}^{*})=A({\bf x},{\bf s},\mbox{\boldmath $\psi$},c,c_{t_0}^{})$,
 where $({\bf x}^*,{\bf s}^*,\mbox{\boldmath ${\psi}^*$},c^{*},c_{t_0}^{*})=$
 RHS of equations \eqref{4.3.11}-\eqref{4.3.15}.
 Let us prove that $A$ maps $\Omega$ into itself. Indeed,
 \[
 |x_{i}^*-X_{i,0}|\leq M_iT\leq h_{x,i},\ \ i=1,...,n,
 \]
 \[
 |s_{j}^*-S_{j}^*|\leq M_{s,j}T^2\leq h_{s,j}, \ \ 
 |{\psi}_{i}^*-\Psi_i^*|\leq M_{\psi,i}T^2\leq h_{\psi,i},\ \
 \ i=1,...,n, \ j=1,...,m,
 \]
 \[
  \ |c^{*}-\Sigma|\leq 2M_{c,1}T^2 \leq h_{c,1},
  \ |c_{t_0}^{*}-\sigma_a|\leq M_{c,2}T\leq h_{c,2}.
 \]
\noindent
 Consider $(\tilde{\bf x},\tilde{\bf s},\tilde{\mbox{\boldmath $\psi$}},  \tilde c, \tilde c_{t_0}^{})\in\Omega \ $ 
  and let
 $(\tilde{\bf x}^*,\tilde{\bf s}^*, \tilde{\mbox{\boldmath $\psi$}}^*, \tilde c^*, \tilde c_{t_0}^{*})
 =A(\tilde{\bf x},\tilde{\bf s},\tilde{\mbox{\boldmath $\psi$}}, \tilde c, \tilde c_{t_0}^{})$.
 It is possible to obtain
 \[
 |x_{i}^{*}-\tilde x_{i}^{*}|\leq \lambda_i T
 ||({\bf x},{\bf s},\mbox{\boldmath $\psi$}, c, c_{t_0}^{})-
 (\tilde{\bf x},\tilde{\bf s},\tilde{\mbox{\boldmath $\psi$}},\tilde c, \tilde c_{t_0}^{})||,\ i=1,...,n,
 \]
 \[
 |s_{j}^{*}-\tilde s_{j}^{*}|\leq
 \lambda_{s,j}T^2||({\bf x},{\bf s},\mbox{\boldmath $\psi$}, c, c_{t_0}^{})-
 (\tilde{\bf x},\tilde{\bf s},\tilde{\mbox{\boldmath $\psi$}}, \tilde c, \tilde c_{t_0}^{})||,
 \ j=1,...,m,
 \]
  \[
 |{\psi}_{i}^{*}-\tilde \psi_{i}^{*}|\leq
 \lambda_{\psi,i}T^2||({\bf x},{\bf s},\mbox{\boldmath $\psi$}, c, c_{t_0}^{})-
 (\tilde{\bf x},\tilde{\bf s},\tilde{\mbox{\boldmath $\psi$}}, \tilde c, \tilde c_{t_0}^{})||,
 \ i=1,...,n,
 \]
 \[
 |c^{*}-\tilde c^{*}|\leq
 (\lambda_{c,1}+\lambda_{c,2})T^2||({\bf x},{\bf s},\mbox{\boldmath $\psi$}, c, c_{t_0}^{})-
(\tilde{\bf x},\tilde{\bf s},\tilde{\mbox{\boldmath $\psi$}},  \tilde c, \tilde c_{t_0}^{})||,
 \]
 \[
 |c_{t_0}^{*}-\tilde c_{t_0}^{*}|\leq
 \lambda_{c,3}T||({\bf x},{\bf s},\mbox{\boldmath $\psi$},c, c_{t_0}^{})-
 (\tilde{\bf x},\tilde{\bf s},\tilde{\mbox{\boldmath $\psi$}}, \tilde c, \tilde c_{t_0}^{})||.
 \]
\noindent
 Therefore,
 \[
 ||({\bf x}^{*},{\bf s}^{*}, \mbox{\boldmath ${\psi}^*$}, c^{*}, c_{t_0}^{*})-
 (\tilde{\bf x}^{*},\tilde{\bf s}^{*}, \tilde{\mbox{\boldmath $\psi$}}^*, \tilde c^{*}, \tilde c_{t_0}^{*})||
 \leq  
 \Lambda
 ||({\bf x},{\bf s},\mbox{\boldmath $\psi$}, c, c_{t_0}^{})-
 (\tilde{\bf x},\tilde{\bf s},\tilde{\mbox{\boldmath $\psi$}}, \tilde c, \tilde c_{t_0}^{})||,
 \]
 where
 \[
 \Lambda=aT^2+bT.
 \]
 According to \eqref{4.3.16} $\Lambda<1$, proving Theorem \ref{thm}.
\end{proof}

\begin{rmk}\label{rmk2}
The mathematical model \eqref{2.3.1} -\eqref{2.3.10} describes the initial phase of biofilm formation, when the detachment flux $\sigma_d$ is negligible compared to the attachment flux $\sigma_a$. When $\sigma_d$ is not negligible and the net exchange flux at the interface biofilm-bulk liquid is positive ($\sigma_a-\sigma_d>0$), Eq. \eqref{2.3.3} and its initial condition \eqref{2.3.4} can be rewritten as follows:
\begin{equation}              \label{4.1}
\dot R(t) = u(R(t),t) + \sigma_a(\mbox{\boldmath $\Psi$}^*)- \sigma_d(R(t)), \ \sigma_a-\sigma_d>0, \ t>0, \ R(0)=0.
\end{equation}

Consider Eq. \eqref{4.1}$_1$ written for $t = t_0$: 
\begin{equation}              \label{4.3}
\dot R(t_0)= \sigma_a(\mbox{\boldmath $\Psi$}^*(t_0))-  \sigma_d(R(t_0))+ u(R(t_0),t_0).
\end{equation}

Substituting Eq. \eqref{3.4} in Eq. \eqref{4.3} it follows:
\begin{equation}              \label{4.4}
\begin{aligned}
\dot R(t_0)&= \sigma_a(\mbox{\boldmath $\Psi$}^*(t_0))- \sigma_d(c(t_0,t_0))\\
&+ \frac{1}{c^2(t_0,t_0)}\int_{0}^{t_0} c^2(\tau,t_0) G({\bf X}(c(\tau,t_0),t_0),{\bf S}(c(\tau,t_0),t_0),\mbox{\boldmath $\Psi$}(c(\tau,t_0),t_0)) \frac{\partial }{\partial \tau} c(\tau,t_0) d\tau.
\end{aligned}
 \end{equation}

Integrating over $(0,t_0)$ it follows:
\begin{equation}              \label{4.5}
\begin{aligned}
 R(t_0) &= \int_{0}^{t_0} [\sigma_a(\mbox{\boldmath $\Psi$}^*(\theta))- \sigma_d(c(\theta,\theta))] d\theta \\
 &+ \int_{0}^{t_0} \frac{1}{c^2(\theta,\theta)} d\theta \int_{0}^{\theta} c^2(\tau,\theta) G({\bf X}(c(\tau,\theta),\theta),{\bf S}(c(\tau,\theta),\theta),\mbox{\boldmath $\Psi$}(c(\tau,\theta),\theta))  \frac{\partial }{\partial \tau} c(\tau,\theta) d\tau,
\end{aligned}
\end{equation}
where the initial condition has been used. Equation \eqref{4.5} is the desired integral equation for $R$ in characteristic coordinates in the case of attachment regime ($\sigma_a-\sigma_d>0$) when $\sigma_d$ is not negligible. 

Substituting the new expression \eqref{4.5} of $R(t_0)$ in Eq. \eqref{3.10} yields:
 \begin{equation}               \label{4.6}
\begin{aligned}
c(t_0,t) &= \int_{0}^{t_0} [\sigma_a(\mbox{\boldmath $\Psi$}^*(\theta))- \sigma_d(c(\theta,\theta))]d\theta \\
&+ \int_{0}^{t_0} \frac{1}{c^2(\theta,\theta)} d\theta \int_{0}^{\theta} c^2(\tau,\theta) G({\bf X}(c(\tau,\theta),\theta),{\bf S}(c(\tau,\theta),\theta),\mbox{\boldmath $\Psi$}(c(\tau,\theta),\theta)) \frac{\partial c}{\partial \tau} (\tau,\theta) d\tau \\
&+ \int_{t_0}^{t} \frac{1}{c^2(t_0,\theta)} d\theta \int_{0}^{t_0} c^2(\tau,\theta)  G({\bf X}(c(\tau,\theta),\theta),{\bf S}(c(\tau,\theta),\theta),\mbox{\boldmath $\Psi$}(c(\tau,\theta),\theta)) \frac{\partial c}{\partial \tau} (\tau,\theta) d\tau.
\end{aligned}
\end{equation}
   
Equation \eqref{4.6} is the desired integral equation for the characteristics in the case of attachment regime ($\sigma_a-\sigma_d>0$) when $\sigma_d$ is not negligible. 

In addition, the integral equation for $\partial c/\partial t_0$ can be obtained by deriving Eq. \eqref{4.6} and reads:
\begin{equation}         \label{4.7}
\begin{aligned}
\frac{\partial }{\partial t_0} c(t_0,t) &= \sigma_a(\mbox{\boldmath $\Psi$}^*(t_0))- \sigma_d(c(t_0,t_0))\\
&+ \int_{t_0}^{t} G({\bf X}(c(t_0,\theta),\theta),{\bf S}(c(t_0,\theta),\theta),\mbox{\boldmath $\Psi$}(c(t_0,\theta),\theta)) \frac{\partial }{\partial t_0} c(t_0,\theta) d\theta.
\end{aligned}
 \end{equation}

An existence and uniqueness result for the integral system \eqref{4.5}, \eqref{4.6}, \eqref{4.7}, \eqref{3.17}, \eqref{3.23} and \eqref{3.29} which models the biofilm dynamics under the hypothesis $\sigma_a-\sigma_d>0$ can be obtained as generalization of the reasonings presented in Section \ref{n4}. 
\end{rmk}

\section{Conclusions} \label{n5}
\
In this work, a free boundary value problem governing the growth of a multispecies granular biofilm is discussed. The presented model accounts for the dynamics of sessile microbial species, dissolved substrates and planktonic cells. It properly reproduces the evolution of a granular biofilm, starting from the initial formation mediated by the attachment process and including the establishment and growth of new species thanks to the invasion phenomena. Modelling the initial phase of granular biofilm formation allows to describe the biofilm growth without arbitrarily fixing the initial size and microbial composition of the granular biofilm.

The work presents for the first time the qualitative analysis of the spherical free boundary problem for the initial phase of multispecies granular biofilm growth. The biological context is related to biofilm genesis, where the attachment is the only relevant exchange flux between biofilm and bulk liquid. An existence and uniqueness result has been provided for this spherical free boundary value problem by using the method of characteristics and the fixed point theorem. The existence and uniqueness result has been obtained for an arbitrary number of microbial species $n$ and dissolved substrates $m$, with non linear reaction terms. Moreover, all theorem  hypotheses are not suggested by mathematical artefacts, but they are mostly qualitative and naturally derived from biological considerations. The influence of the environmental conditions on the attachment process represents a key factor in the mathematical modelling of the granular biofilm genesis. This aspect has not been considered in this work, as it requires further investigations. Furthermore, the existence and uniqueness result has been derived in the case of negligible $\sigma_d$ which models the initial phase of biofilms formation. Further work may address the case of detachment regime with $\sigma_a-\sigma_d<0$.

\section*{Acknowledgements}

Fabiana Russo's research activity has been conducted in the context of D.D. n. 155 on 17 May 2018 additional PhD fellowships for 2018/2019 academic year, course XXXIV within the framework of POR Campania FSE 2014-2020 ASSE III - Specific objective 14 Action 10.5.2 - Public notice "Innovative PhD with industrial characterization".

This paper has been performed under the auspices of the G.N.F.M. of I.N.d.A.M.

\section*{Declaration of Interest}

Declarations of interest: none.

\clearpage
\bibliographystyle{unsrt}      
\bibliography{Bibliography}

\end{document}